\documentclass[11pt,reqno,a4paper]{amsart}
\usepackage[english]{babel}
\usepackage{enumitem}
\usepackage[all]{xy}
\SelectTips{cm}{11}

\makeatletter
\newlength{\@thlabel@width}%
\newcommand{\thmenumhspace}{\settowidth{\@thlabel@width}{\upshape(i)}\sbox{\@labels}{\unhbox\@labels\hspace{\dimexpr-\leftmargin+\labelsep+\@thlabel@width-\itemindent}}}
\makeatother

\theoremstyle{definition}
\newtheorem{definition}{Definition}[section]
\newtheorem{conjecture}[definition]{Conjecture}
\theoremstyle{plain}
\newtheorem{theorem}[definition]{Theorem}
\newtheorem{proposition}[definition]{Proposition}
\newtheorem{corollary}[definition]{Corollary}
\newtheorem{lemma}[definition]{Lemma}
\theoremstyle{remark}
\newtheorem{remark}[definition]{Remark}

\title[Exterior and interior Whitehead products]{On the higher order exterior and interior Whitehead products}

\author{Marek Golasi\'nski}
\address{Institute of Mathematics, Casimir the Great University, pl.\ Weyssenhoffa 11 85--072, Bydgoszcz, Poland}
\email{marek@ukw.edu.pl}

\author[Thiago de Melo]{Thiago de Melo\textsuperscript{*}}
\address{Instituto de Geoci\^encias e Ci\^encias Exatas\\ UNESP--Univ Estadual Paulista\\ Av.~24A, 1515, Bela Vista. CEP 13.506--900. Rio Claro--SP, Brazil}
\email{tmelo@rc.unesp.br}
\thanks{\textsuperscript{*}Supported by FAPESP 2014/21926-1}


\begin{document}\baselineskip=1.5em

\begin{abstract}We extend the notion of the exterior Whitehead product for maps $\alpha_i:\Sigma A_i \to X_i$ for $i=1,\dots,n$,
where $\Sigma A_i$ is the reduced suspension of $A_i$ and then, for the interior product with $X_i=J_{m_i}(X)$ as well.
The main result stated in Theorem~\ref{thm:main} generalizes \cite[Theorem~1.10]{hardie3} and concerns
to the Hopf invariant of the generalized Hopf construction.

We close the paper applying the Gray's construction $\circ$ (called the Theriault product) to a sequence $X_1,\dots,X_n$ of simply connected co-$H$-spaces to obtain
a higher Gray--Whitehead product map \[w_n:\Sigma^{n-2}(X_1\circ\dots\circ X_n)\to T_1(X_1,\dots,X_n),\] where $T_1(X_1,\dots,X_n)$ is the fat wedge  of $X_1,\dots,X_n$.
\end{abstract}

\maketitle

\section*{Introduction} Porter \cite{porter} has generalized the Hardie's construction from \cite{hardie1} and introduced the notion of
the $n$\textsuperscript{th} order generalized Whitehead product of maps $\alpha_i:\Sigma A_i\to X$
for $i=1,\dots ,n$, with $n\geq 2$ which is very useful in many mathematical constructions. For example, given a simplicial complex $K$ on $n$ vertices,
Davis and Januszkiewicz \cite{dj} have associated two fundamental objects of toric topology: the moment--angle complex
$\mathcal{Z}_K$ and the Davis--Januszkiewicz space $DJ_K$. The homotopy fibration sequence
\[\mathcal{Z}_K\xrightarrow{\tilde{\omega}}DJ_K\to\textstyle\prod\limits^n_{i=1}\mathbb{C}P^\infty\] and its generalization have been studied
in \cite{gt} and \cite{kish}, respectively to show that $\tilde{\omega} : \mathcal{Z}_K\to DJ_K$ is a sum of higher and iterated Whitehead products
for appropriate complexes $K$.

Next, let $\mathcal{F}(\mathbb{R}^{n+1},m)$ be the Euclidean ordered configuration space.
By  Salvatore \cite[Theorem~7]{salvatore}, the homotopy type of $\mathcal{F}(\mathbb{R}^{n+1},m)$ for $n \geq 2$
admits a minimal cellular model
\[\ast=X_0\subseteq X_n\subseteq X_{2n}\subseteq\dots \subseteq X_{mn}\]
whose cells are attached via higher order Whitehead products.

Hardie \cite{hardie1} has made use of the reduced product spaces $J_m(X)$, defined by James in \cite{james}, to study the interior Whitehead product
of maps $\alpha_i:\mathbb{S}^{k_i}\to J_{m_i}(X)$ with $k_i,m_i\geq 1$ for $i=1,\dots,n$ as an element  $\langle\alpha_1,\dots,\alpha_n\rangle\in\pi_{k-1}(J_m(X))$,
where $k=k_1+\dots+k_n$ and
$m=m_1+\dots+m_n-\min\limits_{1\leq i\leq n}\{m_i\}$. In addition, by means of the generalized Hopf construction from \cite{hardie3},
for a given map $F:\mathbb{S}^{k_1}\times\dots\times\mathbb{S}^{k_n}\to J(X)$ strongly of type $(\alpha_1,\dots,\alpha_n)^{m-1}$, Hardie has defined
an element $c(F)\in \pi_{k+1}(\Sigma X)$ and in particular, an element of order $p$ in $\pi_{2p}(\mathbb{S}^3)$ analysed in \cite{hardie1}.

The main result on the triple spherical Whitehead product from \cite{hardie} has been generalized in \cite{marek-thiago} into suspensions.
After necessary prerequisites exhibited in Section~\ref{sec.james}, we extend the notion of the exterior Whitehead product for maps $\alpha_i:\Sigma A_i \to X_i$ for $i=1,\dots,n$,
where $\Sigma A_i$ is the reduced suspension of $A_i$ and then, for the interior product with $X_i=J_{m_i}(X)$ as well. Next, some properties of these products are presented in Section~\ref{sec.ext.int}.

James has shown \cite{james2} that, with the exception of the toric constructions of Toda \cite{toda}, the usual procedures for the construction
of generators of homotopy groups of spheres give rise to no more elements that can be obtained by the Hopf construction together with the operation
of composition (the Toda bracket).
We follow Hardie \cite{hardie3} to adapt the necessary results on generalized
Hopf construction. First, we list in Section~\ref{subsec.separation} some properties of the separation map $d(u,v):\Sigma A\to Y$ of maps $u,v:C_f\to Y$ defined on the mapping cone $C_f$ and studied in
\cite{tsuchida}, for a given map $f:A\to X$.
Then, the main result stated in Theorem~\ref{thm:main} generalizes \cite[Theorem 1.10]{hardie3} and concerns
to the Hopf invariant of the generalized Hopf construction.

Recently, Gray has defined in \cite{gray} a functor $\circ$ (called the Theriault product) in the category $\mathcal{CO}$ of simply-connected co-$H$-spaces
and co-$H$-maps, and also a natural transformation $X\circ Y \to X\vee Y$ generalizing the Whitehead product map.

In Section~\ref{subsec.misc}, we close the paper making use of the Gray's construction \cite{gray} to a sequence $X_1,\dots,X_n$ of simply-connected co-$H$-spaces and a fiber sequence
from \cite{porter3} (as in the paper \cite{gt}) to obtain a higher Gray--Whitehead product map $w_n:\Sigma^{n-2}(X_1\circ\dots\circ X_n)\to T_1(X_1,\dots,X_n)$, where $T_1(X_1,\dots,X_n)$ is
the fat wedge of $X_1,\dots,X_n$. The map $w_n$ is used to introduce higher order Whitehead product for maps
defined on co-$H$-spaces. Its basic properties and applications extending those from \cite{gray} will be presented in a forthcoming paper.

Finally, we discuss some connections, via the interior Whitehead product, of the James construction with the symmetric product and then with projective spaces $\mathbb{F}P^n$,
for $\mathbb{F}=\mathbb{R},\mathbb{C}$ or $\mathbb{H}$ as well.

\subsubsection*{Acknowledgement} The authors would like to warmly thank Michael J.\ Hopkins for his helpful conversations
on the maps $J_n(\mathbb{S}^2)\to \mathbb{C}P^n$ studied in Section~\ref{subsec.misc}.
The first author is indebted to the Institute for Mathematical Sciences, the National University of Singapore for a support to attend the Conference \textit{Combinatorial and Toric Homotopy} (Singapore, August 23--29, 2015) in honour of Frederick Cohen and present main results of this paper.

\section{James construction and Hopf invariant}\label{sec.james} In this work all spaces are based, have the homotopy type of a $CW$-complex and we do not distinguish between a based map and its homotopy class.
Given a well-pointed Hausdorff space $(X,\ast)$, the James construction $J(X)$ is the free associative monoid on $X$ with $\ast$ as unit.
More precisely, from \cite{james}, for each $n\geq 1$ let $J_n(X)$ be the quotient of $X\times \overset{\times n}{\dotsb{}}\times X$,
where $(x_1,\dots,x_n)\sim (x'_1,\dots,x'_n)$ provided they are equal after removing any occurrence of $\ast$.  Then,
$J(X)=\operatorname{\underrightarrow{\lim}} J_n(X)$, where $X=J_1(X)\subseteq J_2(X)\subseteq\dots\subseteq J_n(X)\subseteq J_{n+1}(X)\subseteq\cdots$
is the James filtration. Given $f:X\to Y$, there are maps $J_n(f): J_n(X)\to J_n(Y)$ for $n\ge 1$ and
$J(f) : J(X)\to J(Y)$. Further, there are natural multiplication maps \[\mu_{m,n}(X): J_m(X)\times J_n(X)\to J_{m+n}(X)\]
and \[\mu(X) : J(X)\times J(X)\to J(X)\] defined by the juxtaposition.

Let $\eta_X : X\to\Omega\Sigma X$ and $\varepsilon_X : \Sigma\Omega X\to X$ be the canonical maps determined by the pair of adjoint functors $\Sigma$
and $\Omega$.

By \cite[Chapter~VII]{wh}, there is a canonical multiplicative extension
\[\xymatrix@C=1.5cm{X\ar[r]^{\eta_X} \ar@{^{(}->}[d] & \Omega\Sigma X\ar[d]^\simeq\\
J(X)\ar@{-->}[r]^{\bar{\phi}_X} & \Omega'\Sigma X \rlap{,} }  \]
where $\Omega'X$ denotes the Moore loop space of $X$.
Consequently, there is a homotopy equivalence $\phi_X:J(X)\xrightarrow{\simeq} \Omega\Sigma X$ with the adjoint map $\psi_X:\Sigma J(X)\to \Sigma X$.
Writing $\pi(X,Y)$ for the set of homotopy classes of maps from $X$ to $Y,$ we have a commutative diagram
\[\xymatrix@C=1.5cm{
\pi(A,J(X)) \ar[r]^-{\phi_{X\ast}} \ar[d]_{\Sigma} & \pi(A,\Omega\Sigma X)  \ar[d]^{\operatorname{Adj}}_\approx \\
\pi(\Sigma A,\Sigma J(X)) \ar[r]^-{\psi_{X\ast}} & \pi(\Sigma A,\Sigma X)} \]
for any pointed space $A$.
Also, by \cite{james} there is a homotopy equivalence \[\Sigma\Omega\Sigma X \xrightarrow{\simeq} \textstyle\bigvee\limits_{m\geq 1} \Sigma X^{\wedge m},\]
where $X^{\wedge m}$ denotes the iterated smash product $X\wedge\overset{\times m}{\dotsb{}}\wedge X$. Then, for each $m\geq 1$ consider the composition \[ \Sigma\Omega\Sigma X \xrightarrow{\simeq}  \textstyle\bigvee\limits_{m\geq 1} \Sigma X^{\wedge m} \xrightarrow{p_m}
\Sigma X^{\wedge m} \] where $p_m: \bigvee\limits_{m\geq 1}\Sigma X^{\wedge m}\to \Sigma X^{\wedge m}$ is the projection map.

By the adjointness, we obtain the \textit{$m$\textsuperscript{th} Hopf--James invariant} \[\mathcal{H}_m : \Omega \Sigma X \to \Omega\Sigma X^{\wedge m} \] which
induces (again by the adjointness)
a map \begin{equation}h_m: \pi(\Sigma A,\Sigma X)\to \pi(\Sigma A,\Sigma X^{\wedge m})\label{eq.gen.hopf}\end{equation} for any pointed space $A$.

Recall that:
\begin{enumerate}[label={\textup{\arabic*.}}]
\item $\varepsilon_{\Sigma X}\Sigma\eta_X=\mathrm{id}_{\Sigma X}$;
\item $\psi_X=\varepsilon_{\Sigma X}\Sigma\phi_X$;
\item $\phi_X=(\Omega\psi_X)\eta_{J(X)}$;
\item $\phi_Xj_X=\eta_X$;
\item $\psi_X\Sigma j_X=\mathrm{id}_{\Sigma X}$,
\end{enumerate}
for the embedding map $j_X : X\hookrightarrow J(X)$.

Now, let $q : J(X)\to J(X^{\wedge m})$ be the combinatorial extension (see e.g., \cite{bs} for its construction) of the quotient map
$q_m : (J_m(X),J_{m-1}(X))\to (X^{\wedge m},\ast)$ which collapses $J_{m-1}(X)$ to $\ast$ for $m\geq 1$. Given a map $f : \Sigma A\to \Sigma X$,
we write $\tilde{f}=(\Omega f)\eta_A$ and get the following formula for the Hopf--James invariant:
\begin{equation}\label{for}
h_m(f)=\varepsilon_{\Sigma X^{\wedge m}}(\Sigma\phi_{X^{\wedge m}}q\phi_X^{-1})(\Sigma\Omega f)\Sigma \eta_A=\psi_{X^{\wedge m}}(\Sigma q\phi_X^{-1})\Sigma\tilde{f}.
\end{equation}

\section{Generalized exterior and interior products}\label{sec.ext.int}
Given the spaces $X_1,\dots,X_n$, write  $\underline{X}=(X_1,\dots,X_n)$. Let $T_i(\underline{X})$ be the subspace of $X_1\times\dots\times X_n$ consisting of those points with
at least $i$ coordinates at base points. In particular, $T_0(\underline{X})=X_1\times\dots\times X_n$ and $T_{n-1}(\underline{X})=X_1\vee\dots\vee X_n$. 
Denote by $\Lambda(\underline{X})$ the smash product $X_1\wedge\dots\wedge X_n$. Notice that for or maps $\alpha_i : X_i\to Y_i$  with $i=1,\dots,n$, the induced 
map $\alpha_1\times\dots\times \alpha_n : T_0(\underline{X})\to T_0(\underline{Y})$ restricts to maps
\[T_i(\underline{\alpha}) : T_i(\underline{X})\to T_i(\underline{Y})\]
for $i=1,\dots,n$, where $\underline{\alpha}=(\alpha_1,\dots,\alpha_n)$.

Given a space $X$, write $CX$ for the cone of $X$ and notice that there is a canonical embedding $\iota_X: X\hookrightarrow CX$.
According to \cite[Theorem~(2.3)]{porter} there is a (up to homotopy) pushout
\begin{equation}
\begin{aligned}
\xymatrix{
\Sigma^{n-1}\Lambda(\underline{A}) \ar@{^{(}->}[d] \ar[r]^-{\omega_n} & T_1(\underline{\Sigma A})\ar@{^{(}->}[d] \\
C\Sigma^{n-1}\Lambda(\underline{A}) \ar[r]^-{\Omega_n} & T_0(\underline{\Sigma A})\rlap{.}
}
\end{aligned}\label{diag.omega}
\end{equation}
Hence, we have the cofibre sequence \[\Sigma^{n-1}\Lambda(\underline{A})\xrightarrow{\omega_n}T_1(\underline{\Sigma A})\hookrightarrow T_0(\underline{\Sigma A})\]
and write
\[\nu_{\omega_n}: T_1(\underline{\Sigma A})\cup_{\omega_n} C\Sigma^{n-1}\Lambda(\underline{A})\xrightarrow{\simeq} T_0(\underline{\Sigma A})\]
for a homotopy equivalence. This yields the commutative (up to homotopy) square
\begin{equation}
\begin{gathered}
\xymatrix{
C\Sigma^{n-1}\Lambda(\underline{A}) \ar[d]_{p} \ar[r]^-{\Omega_n} & T_0(\underline{\Sigma A}) \ar[d]^{s} \\
\Sigma^{n}\Lambda(\underline{A}) \ar[r]^-{\lambda} & \Lambda(\underline{\Sigma A})\rlap{,} }
\end{gathered}\label{diag.lambda}
\end{equation}
where $\lambda:\Sigma^n\Lambda(\underline{A})\xrightarrow{\simeq} \Lambda(\underline{\Sigma A})$ is a homotopy equivalence.
\begin{definition}\label{def.ext.prod}The \textit{exterior Whitehead product} $\{\alpha_1,\dots,\alpha_n\}$ of maps $\alpha_i: \Sigma A_i \to X_i$ for
$i=1,\dots,n$ with $n\geq 2$ is the composition \[T_1(\underline{\alpha})  \omega_n :\Sigma^{n-1}\Lambda(\underline{A})\to T_1(\underline{X}).\]
\end{definition}
If $A_i=\mathbb{S}^{m_i}$ is the $m_i$-sphere with $m_i\geq 1$ for $i=1,\dots,n$ then $\{\alpha_1,\dots,\alpha_n\}$ has been defined by Hardie in \cite{hardie1}. In the sequel, we refer to such a product as the \textit{spherical} one.

From the above, we derive that $\{\alpha_1,\dots,\alpha_n\}=0$ if and only if there is a map $T' : T_0(\underline{\Sigma A})\to T_1(\underline{X})$
such that the triangle
\[\xymatrix@C=1.5cm{T_1(\underline{\Sigma A})\ar@{^{(}->}[d] \ar[r]^-{T_1(\underline{\alpha})} & T_1(\underline{X}) \\
T_0(\underline{\Sigma A}) \ar[ru]_-{T'}
}
\]
commutes (up to homotopy).

\begin{proposition}Let $\alpha_i:\Sigma A_i \to X_i$ be maps for $i=1,\dots,n$.
\begin{enumerate}[label={\textup{(\arabic*)}}]
\item If $\alpha_{i_0}=0$ for some $1\le i_0\le n$ then
 \[\{\alpha_1,\dots,\alpha_{i_0-1},0,\alpha_{i_0+1},\dots,\alpha_n\}=0;\]

\item  $\Sigma_\ast \{\alpha_1,\dots,\alpha_n\}=0$, where $\Sigma_\ast$ is the suspension homomorphism.
\end{enumerate}
\end{proposition}

\begin{proof} (1): In virtue of Proposition~\ref{prop.ext}\ref{prop.ext.permut}, we can suppose that $i_0=1$.
Define $T': T_0(\underline{\Sigma A}) \to T_1(\underline{X})$ by $T'(a_1,\dots ,a_n)= T_1(\underline{\alpha})(\ast,a_2,\dots ,a_n)$
for any $(a_1,\dots,a_n)\in T_0(\underline{\Sigma A})$. Then, $T'$ is an extension (up to homotopy) of
$T_1(\underline{\alpha})$ and the result follows.

(2): this is a direct consequence of \cite[Corollary~(4)]{porter}.
\end{proof}

We present below some further straightforward properties of the exterior Whitehead product and follow \cite{hardie1} to generalize 
the interior one. First, in view of \cite[Definition~(2.10)]{porter}, we say that two maps $f,g:T_1(\underline{\Sigma A})\to X$ are 
\textit{compatible off} the $i$\textsuperscript{th} coordinate if they coincide on $T_0^{(i)}(\underline{\Sigma A})$, where $T_0^{(i)}(\underline{\Sigma A})=\Sigma A_1\times\dots\times \Sigma A_{i-1}\times \Sigma A_{i+1}\times \dots\times \Sigma A_n$ is canonically embedded into $T_1(\underline{\Sigma A})$ for $i=1,\dots,n$.

In addition, if $A_{i_0}$ is a co-$H$-group with a comultiplication $\nu_{i_0} :A_{i_0}\to A_{i_0}\vee A_{i_0}$ then we follow \cite[Definition~(2.11)]{porter}
to define:
\begin{multline*}
(f+^{(i_0)}g)((t_1,a_1),\dots ,(t_{i_0},a_{i_0}),\dots ,(t_n,a_n))\\
=\begin{cases}
f((t_1,a_1),\dots ,(t_{i_0},a'_{i_0}),\dots ,(t_n,a_n)), & \text{if $\nu_{i_0}(a_{i_0})=(a'_{i_0},\ast)$}; \\[1ex]
g((t_1,a_1),\dots ,(t_{i_0},a''_{i_0}),\dots ,(t_n,a_n)), & \text{if $\nu_{i_0}(a_{i_0})=(\ast,a''_{i_0})$.}
\end{cases}
\end{multline*}

Suppose that there are maps
$T_1(\underline{\alpha}'),T_1(\underline{\alpha}''):T_1(\underline{\Sigma A})\to T_1(\underline{X})$ with $\underline{\alpha}'=(\alpha_1,\dots,\alpha'_{i_0},\dots,\alpha_n)$
and $\underline{\alpha}''=(\alpha_1,\dots,\alpha''_{i_0},\dots,\alpha_n)$ which are clearly compatible off the $i_0$\textsuperscript{th} coordinate.
Then, $T_1(\underline{\alpha}')+^{(i_0)}T_1(\underline{\alpha}'')$ is defined and by means of an appropriate version of \cite[Theorem~(2.13)]{porter}, we can state:

\begin{proposition}If $\alpha_i : \Sigma A_i\to X_i$ are maps for $i=1,\dots,n$ and $A_{i_0}$ is a co-$H$-group for some $1\le i_0\le n$ then the exterior Whitehead product satisfies:
\begin{enumerate}[label={\textup{(\arabic*)}}]
\item\label{prop.linear} \(\{\alpha_1,\dots,\alpha'_{i_0},\dots,\alpha_n\}+\{\alpha_1,\dots,\alpha''_{i_0},\dots,\alpha_n\}= \{\alpha_1,\dots,\alpha'_{i_0}+\alpha''_{i_0},\dots,\alpha_n\}\);
\item\label{prop.linear.minus} $\{\alpha_1,\dots,-\alpha_{i_0},\dots,\alpha_n\}=-\{\alpha_1,\dots,\alpha_n\}$.
\end{enumerate}\label{prop.lin.exterior}
\end{proposition}

Certainly, item \ref{prop.linear.minus} above is easily deduced from item \ref{prop.linear}.

We note that Proposition~\ref{prop.lin.exterior} has been shown in \cite{hardie1} for the spherical case using the star product $\star$ studied in \cite{blakers}.

Let $\theta_i: X_i\hookrightarrow T_2(\underline{X})$ and $\Psi_1^{(i)}: T_1^{(i)}(\underline{X})\hookrightarrow T_2(\underline{X})$ for $i=1,\dots,n$
be the canonical embeddings. Given $\alpha_i:\Sigma A_i\to X_i$, we can consider the compositions $\theta_i\alpha_i$ and $\Psi_1^{(i)} \{\alpha_1,\dots,\alpha_n\}^{(i)}$,
explicitly defined by \[ \Sigma A_i \xrightarrow{\alpha_i} X_i \lhook\joinrel\xrightarrow{\theta_i} T_2(\underline{X}) \] and \[  \Sigma^{n-2}\Lambda^{(i)} (\underline{A}) \xrightarrow{\omega_n^{(i)}} T_1^{(i)}(\underline{\Sigma A})  \xrightarrow{T_1^{(i)}(\underline{\alpha})} T_1^{(i)}(\underline{
X}) \lhook\joinrel\xrightarrow{\Psi_1^{(i)}} T_2(\underline{X}) \]
for $i=1,\dots,n$, respectively.

\begin{remark}\label{rem.permutation}In the sequel, given a permutation $\sigma\in S_n$ of the set $\{1,\dots,n\}$ we write \[ \hat\sigma :  A_1\wedge\dots\wedge A_n \xrightarrow{\approx} A_{\sigma(1)}\wedge\dots\wedge A_{\sigma(n)} \]  for the
associated homeomorphism.
\end{remark}

Following the results from \cite[Section~2]{hardie1} on the spherical exterior Whitehead product and the generalized \cite[Lemma~4.1]{marek-thiago}
boundary Nakaoka--Toda operation formula \cite[Lemma~(1.2)]{nakaoka-toda}, we may state:

\begin{proposition}\label{prop.ext}
\begin{enumerate}[label=\textup{(\arabic*)}]\thmenumhspace
\item\label{prop.ext.permut} If $\sigma\in S_n$ is a permutation then \[(\Sigma^{n-1}\hat{\sigma})^*\{\alpha_{\sigma(1)},\dots,\alpha_{\sigma(n)}\}=\bar{\sigma}\{\alpha_1,\dots,\alpha_n\}\]
where $\bar{\sigma}:T_1(\underline{X})\xrightarrow{\approx} T_1(\underline{\sigma X})$ is the homeomorphism induced by $\sigma$.
\item\label{prop.ext.jacobi} The exterior Whitehead product satisfies the Jacobi identity \[ \textstyle\sum\limits_{i=1}^n(\Sigma^{n-2} \hat\sigma_i)^\ast [\theta_i \alpha_i, \Psi_1^{(i)} \{\alpha_1,\dots,\alpha_n\}^{(i)}]=0 \]
in $ \pi(\Sigma^{n-2}\Lambda (\underline{A}), T_2(\underline{X}))$ where $\hat\sigma_i:A_i\wedge\Lambda^{(i)}(\underline{A})\xrightarrow{\approx} \Lambda(\underline{A})$ is induced by an appropriate $\sigma_i\in S_n$ for $i=1,\dots,n$.
\item  If $f_i: X_i\to Y_i$ are maps for $i=1,\dots,n$ then
\[\{f_1\alpha_1,\dots,f_n\alpha_n\}=T_1(\underline{f})\{\alpha_1,\dots,\alpha_n\},\]
where $T_1(\underline{f}):T_1(\underline{X})\to T_1(\underline{Y})$.
\item Let $\beta_i: B_i\to A_i$ and $\alpha_i:\Sigma A_i\to X_i$ be any maps for $i=1,\dots,n$. Then, the following Whitehead identity holds
\textup{(\textit{cf.\ \cite[(2.4)]{hardie1} and \cite[(3.59)]{whiteheadG}})}: \[\{\alpha_1\Sigma\beta_1,\dots,\alpha_n\Sigma\beta_n\}=\{\alpha_1,\dots,\alpha_n\}\Sigma^{n-1}(\beta_1\wedge\dots\wedge\beta_n).\]
\end{enumerate}
\end{proposition}

The Jacobi identity stated in Proposition~\ref{prop.ext}\ref{prop.ext.jacobi} has been also considered in \cite[Corollary~1.9]{kish} with a different approach.

Let $J_n(X)$ be the $n$\textsuperscript{th} stage of the James construction $J(X)$ of a topological space $X$.
Given $m_i\geq 1$ for $i=1,\dots,n$ with $n\geq 2$, write $\mathbf{m}=(m_1,\dots,m_n)$ and let $\underline{J_{\mathbf{m}}(X)}=(J_{m_1}(X),\dots,J_{m_n}(X))$. Set $m'=m_1+\dots+m_n$ and $m''=m'-\min\limits_{1\leq i\leq n}\{m_i\}$ and note that there is an inclusion
$J_{m''}(X)\hookrightarrow J_{m'}(X)$. Next, consider the canonical multiplication \(\mu_{\mathbf{m}}(X): T_0(\underline{J_{\mathbf{m}}(X)})\to J_{m'}(X)\)
which restricts to a map
\[\mu_{\mathbf{m}}(X)_{|T_1(\underline{J_{\mathbf{m}}(X)})}: T_1(\underline{J_{\mathbf{m}}(X)}) \to J_{m''}(X)\]
commuting the diagram
\[
\xymatrix@C=3cm{
T_1(\underline{J_{\mathbf{m}}(X)}) \ar[r]^-{\mu_{\mathbf{m}}(X)_{|T_1(\underline{J_{\mathbf{m}}(X)})}} \ar@{^{(}->}[d] & J_{m''}(X) \ar@{^{(}->}[d]\\
T_0(\underline{J_{\mathbf{m}}(X)}) \ar[r]^-{\mu_{\mathbf{m}}(X)} & J_{m'}(X)
\rlap{.}}
\]

\begin{definition}\label{def.inner.prod}
The \textit{interior Whitehead product} $\langle\alpha_1,\dots,\alpha_n\rangle$ of maps $\alpha_i:\Sigma A_i \to J_{m_i}(X)$ for $i=1,\dots,n$ is the composition \[ \Sigma^{n-1}\Lambda(\underline{A}) \xrightarrow{\{\alpha_1,\dots,\alpha_n\}} T_1(\underline{J_{\mathbf{m}}(X)}) \xrightarrow{\mu_{\mathbf{m}}(X)_{|T_1(\underline{J_{\mathbf{m}}(X)})}} J_{m''}(X), \] and thus $ \langle\alpha_1,\dots,\alpha_n\rangle\in \pi (\Sigma^{n-1}\Lambda(\underline{A}), J_{m''}(X))$.
\end{definition}

Now, consider the composite maps $\alpha'_i :\Sigma A_i\xrightarrow{\alpha_i} J_{m_i}(X)\hookrightarrow J_{m''}(X)$ for $i=1,\dots,n$.
Then, notice that $\langle\alpha_1,\dots,\alpha_n\rangle$ represents an element of the higher order Whitehead product $[\alpha'_1,\dots,\alpha'_n]$
considered in \cite{porter}. Thus, applying \cite[Theorem~(2.1)]{porter} for any map $f:J_{m''}(X)\to Y$ it follows that
\begin{equation}
f_*\langle\alpha_1,\dots,\alpha_n\rangle \in f_*[\alpha'_1,\dots,\alpha'_n] \subseteq [f\alpha'_1,\dots,f\alpha'_n].
\label{eq.int.nat}
\end{equation}

The following properties of the interior Whitehead product are easily obtained from Proposition~\ref{prop.ext}.

\begin{corollary}\label{cor.interior}
\begin{enumerate}[label=\textup{(\arabic*)}]\thmenumhspace
\item Let $\sigma\in S_n$ be a permutation of the set $\{1,\dots,n\}$. Then \[\langle\alpha_{\sigma(1)},\dots,\alpha_{\sigma(n)}\rangle=(\Sigma^{n-1}\hat\sigma)^*\langle\alpha_1,\dots,\alpha_n\rangle,\]
where $\hat\sigma:\Lambda(\underline{A})\xrightarrow{\approx}
\Lambda(\underline{\sigma A})$ is the associated homeomorphism.

\item\label{cor.interior.jacobi} Denote by $\delta_i:J_{m_i}(X)\hookrightarrow J_{m^*}(X)$ the
inclusions for $i=1,\dots,n$, where $m^*=m'-\min\limits_{j<k}\{m_j+m_k\}$. The interior Whitehead product satisfies the Jacobi identity \[
\textstyle\sum\limits_{i=1}^{n} (\Sigma^{n-2} \hat\sigma_i)^\ast [
\delta_i\alpha_i,  \langle\alpha_1,\dots,\alpha_n\rangle^{(i)}]=0 \] as an element of $\pi(\Sigma^{n-2}\Lambda (\underline{A}), J_{m^*}(X))$.

\item If $f: X\to Y$ is any map then \[\langle J_{m_1}(f)\alpha_1,\dots,J_{m_n}(f)\alpha_n\rangle = J_{m''}(f)\langle\alpha_1,\dots,\alpha_n\rangle.\]

\item The following Whitehead identity holds:
\[\langle\alpha_1\Sigma\beta_1,\dots,\alpha_n\Sigma\beta_n\rangle=\langle\alpha_1,\dots,\alpha_n\rangle\Sigma^{n-1}(\beta_1\wedge\dots\wedge\beta_n).\]
\end{enumerate}
\end{corollary}

\begin{remark}If $n=2$, Corollary~\ref{cor.interior}\ref{cor.interior.jacobi} is the result stated in \cite[Theorem~(5.4)]{ando} and \cite[Theorem~1]{rutter}
(corrected in \cite{rutter1}).
\end{remark}

\section{Generalized Hopf construction}\label{sec.hopf}
Let $\underline{A}=(A_1,\dots,A_n)$ be a $n$-tuple of pointed spaces. From \cite[Satz~19]{puppe} there is a homotopy equivalence
$\Sigma T_0(\underline{A}) \xrightarrow{\simeq} \bigvee\limits_{N} \Sigma\bigwedge\limits_{i\in N} A_i$, where $N$ runs through all
non-empty subsets of the set $\{1,\dots,n\}$. By the other hand, Hardie has constructed in \cite{hardie3} a particular homotopy equivalence which
possesses some useful properties. The Hardie's construction uses a right lexicographic order between some subsets $N$. More precisely, for each $i=1,\dots,n$ let $c_i=\binom{n}{i}$ be the binomial coefficient. For $k=1,\dots,c_i$, denote by $N_{i,k}$ the $k$\textsuperscript{th} subset of cardinality $i$ in the ordered sequence \[ \{1,\dots, i\} <\dots <  \{n-(i-1),\dots,n-1,n\}.\]

Let $W(\underline{A})=\bigvee\limits_{N} \Sigma\bigwedge\limits_{i\in N}  A_i$. Following \cite[(2.2)]{hardie3} and making use of the co-$H$-structure on $\Sigma T_0  (\underline{A})$, we define \[\theta=\textstyle\sum\limits_{i=1}^{n}\sum\limits_{k=1}^{c_i} \theta_{i,k}:\Sigma T_0(\underline{A}) \to W(\underline{A}),\] where  \[\theta_{i,k}: \Sigma  T_0 (\underline{A}) \to \Sigma\textstyle\bigwedge\limits_{j\in N_{i,k}} A_j \hookrightarrow W(\underline{A}) \] is determined by suspending the collapsing map $T_0 (\underline{A}) \to \bigwedge\limits_{j\in N_{i,k}} A_j$  and composing with the inclusion map $\Sigma\bigwedge\limits_{j\in N_{i,k}} A_j \hookrightarrow W(\underline{A})$.

\begin{theorem}[{\cite[Theorem~2.3]{hardie3} (cf.\ \cite[Satz~19]{puppe})}] The map $\theta: \Sigma T_0 (\underline{A})\to W(\underline{A})$ is a homotopy equivalence.
\label{hardie.2.3}
\end{theorem}

Recall that given any based map $F:A_1\times A_2\to Z$, the Hopf construction on $F$ leads to a map $H(F):\Sigma(A_1\wedge A_2)\to \Sigma Z$ which is given
by the composition \[\Sigma(A_1\wedge A_2)\xrightarrow{\delta}\Sigma(A_1\times A_2)\xrightarrow{\Sigma F} \Sigma Z,\] where $\delta:\Sigma(A_1\wedge A_2)\to \Sigma(A_1\times A_2)$
is determined by the canonical section of the cofibration
$\xymatrix{\Sigma(A_1\vee A_2)\ar@{^{(}->}[r] & \Sigma(A_1\times A_2)\ar[r] & \Sigma(A_1\wedge A_1).\ar@/_1pc/@{-->}[l]_{\delta}}$

Denoting by $\imath:\Sigma\Lambda(\underline{A})\hookrightarrow W(\underline{A})$ the obvious inclusion, the composite map
\[\delta :\Sigma\Lambda(\underline{A})\lhook\joinrel\xrightarrow{\imath} W(\underline{A}) \xrightarrow{\theta^{-1}} \Sigma T_0(\underline{A})\]
yields a section of the cofibration
\[\xymatrix{\Sigma T_1(\underline{A})\ar@{^{(}->}[r]&\Sigma T_0(\underline{A})\ar[r]&\Sigma\Lambda(\underline{A})\ar@/_1pc/@{-->}[l]_{\delta}}\]
which leads to a homotopy equivalence \[\Sigma T_1(\underline{A})\vee\Sigma\Lambda(\underline{A})\xrightarrow{\simeq}\Sigma T_0(\underline{A}).\]

For the $n$-tuple $\underline{\Sigma A}$, the map $\delta :\Sigma\Lambda(\underline{\Sigma A})\to \Sigma T_0(\underline{\Sigma A})$ might be also described as follows. Given any map $f : A \to X$, we have a cofibration $X \hookrightarrow C_f \to \Sigma A$, where $C_f$ is the mapping cone of $f$. Write $i_X : X\hookrightarrow C_f$ and $i_{CA}:CA\to C_f$ for the canonical maps,  $\tau_1:C_{\Sigma f}\xrightarrow{\approx} \Sigma C_f$ and $\tau_2:C\Sigma A\xrightarrow{\approx} \Sigma CA$ for the canonical homeomorphisms.
Thus, $\tau_1 i_{C\Sigma A}= (\Sigma i_{CA}) \tau_2:C\Sigma A\to \Sigma C_f$.

Notice that for the constant map $c: A \to\ast$ there is a canonical homeomorphism $\nu_c : C_c\xrightarrow{\approx}\Sigma A$.
Next, given a commutative square
\begin{equation}
\begin{aligned}
\xymatrix{A' \ar[r]^{f'}\ar[d]_{\alpha} & X' \ar[d]^{\beta}\\
A\ar[r]^{f} & X\rlap{,} }
\end{aligned}\label{diag.alpha.beta}
\end{equation}
the universal properties of the mapping cones $C_{f'}$ and $C_{f}$ lead to a map $\gamma(f',f):C_{f'}\to C_{f}$ with
$\gamma(f',f) i_{X'}=i_{X}\beta$ and $\gamma(f',f) i_{CA'}=i_{CA} C\alpha$.
Further, the diagram
\begin{equation}
\begin{aligned}
\xymatrix@C=1.5cm{
X' \ar@{^(->}[r]^{i_{X'}}\ar[d]_{\beta} & C_{f'} \ar[r]^{\pi_{X'}}\ar[d]_{\gamma(f',f)} & \Sigma A'\ar[d]^{\Sigma \alpha} \\
X\ar@{^(->}[r]^{i_{X}} & C_{f} \ar[r]^{\pi_{X}} & \Sigma A}
\end{aligned}\label{diag.gamma}
\end{equation}
commutes, where $\pi_{X'} : C_{f'}\to \Sigma A'$ and $\pi_X : C_f\to \Sigma A$ are the projection maps.

Now, because of a section $\Sigma T_0(\underline{\Sigma A})\to \Sigma T_1(\underline{\Sigma A})$ for the inclusion map
$\Sigma T_1(\underline{\Sigma A})\hookrightarrow \Sigma T_0(\underline{\Sigma A})$, the map
$\Sigma \omega_n:\Sigma^n\Lambda(\underline{A})\to \Sigma T_1(\underline{\Sigma A})$ is trivial (cf.~\cite[Corollary~(4)]{porter}).
Consequently, there is a commutative diagram (up to homotopy)
\begin{equation}\label{d}
\begin{gathered}
\xymatrix@C=1cm{\Lambda(\underline{\Sigma A}) \ar[d]_{\lambda^{-1}} \ar[r]^-{c}& \ast \ar[d]\\
\Sigma^n\Lambda(\underline{A}) \ar[r]^-{\Sigma\omega_n} & \Sigma T_1(\underline{\Sigma A}) \rlap{,} }
\end{gathered}
\end{equation}
where $\lambda : \Sigma^n\Lambda(\underline{A}) \xrightarrow{\simeq} \Lambda(\underline{\Sigma A})$ is given by~\eqref{diag.lambda}.
It follows from diagrams \eqref{diag.alpha.beta} and \eqref{diag.gamma}  that there is a map $\gamma(c,\Sigma\omega_n): C_c \to C_{\Sigma \omega_n}$ which yields \begin{equation}
\tau : \Sigma\Lambda(\underline{\Sigma A})\xrightarrow{\nu_c^{-1}} C_c \xrightarrow{\gamma(c,\Sigma\omega_n)} C_{\Sigma \omega_n} \xrightarrow{\tau_1} \Sigma C_{\omega_n} \xrightarrow{\Sigma\nu_{\omega_n}} \Sigma T_0(\underline{\Sigma A}).
\label{eq.tau}\end{equation}

Hence, we have got:
\begin{proposition}\label{p} The maps $\delta,\tau : \Sigma\Lambda(\underline{\Sigma A})\to \Sigma T_0(\underline{\Sigma A})$ coincide.
\end{proposition}

\medskip

The \textit{generalized Hopf construction} on a based map $F:T_0(\underline{A})\to Z$ is the composition 
\[H(F):\Sigma\Lambda(\underline{A})\lhook\joinrel\xrightarrow{\imath} W(\underline{A}) \xrightarrow{\theta^{-1}} \Sigma T_0(\underline{A}) \xrightarrow{\Sigma F} \Sigma Z.\]
In particular, for $Z=J(X)$ Hardie has defined in \cite{hardie3} the element 
\[c(F)=\psi_X H(F) \in \pi(\Sigma\Lambda(\underline{A}),\Sigma X).\]
By the adjointness, we obtain \[\tilde c(F)\in \pi(\Lambda (\underline{A}),\Omega\Sigma X).\]
From~\eqref{eq.gen.hopf}, for each $m\geq 1$, we have \[h_m(c(F))\in\pi(\Sigma\Lambda(\underline{A}),\Sigma X^{\wedge m}) \]
and by~\eqref{for}, it holds
\begin{equation}\label{forr}
h_m(c(F))=\psi_{X^{\wedge m}}(\Sigma q\phi_X^{-1})\Sigma\tilde{c}(F).
\end{equation}

Given based maps $\alpha_i: A_i\to J(X)$ for $i=1,\dots,n$, let $F'=\mu(X) T_0(\underline{\alpha}):T_0(\underline{A})\to J(X)$.
From \cite[Corollary~3.4]{hardie3}, $c(F'):\Sigma\Lambda(\underline{A})\to \Sigma X$ is trivial.

\subsection{Separation map}\label{subsec.separation}
Given a map $f : A\to X$, suppose that $u, v :C_f \to Y$ are maps such that $u i_X= v i_X$. From \cite[Section~3]{tsuchida}, there is a map $w: \Sigma A \to Y$ defined by
\[w(a,t)=\begin{cases}
u(a,2t), & \text{if $0\leq t\leq \frac{1}{2}$,} \\[1ex]
v(a,2-2t),& \text{if $\frac{1}{2}\leq t\leq 1 $}
\end{cases}\]
for $(a,t)\in \Sigma A$, called the \textit{separation map of $u,v$} and denoted by $d(u,v)$.

\begin{remark}\label{rem.sep.map.proj}
If $p_1,p_2:C A \to \Sigma A$ are given by $p_1(a,t)=(a,\frac{t}{2})$ and $p_2(a,t)=(a, \frac{2-t}{2})$ for $a\in  A$ and $0\leq t\leq 1$, the separation map $w=d(u,v)$ satisfies $w p_1=u i_{C A}$ and $ w p_2=v i_{C A}$.
Further, the map $w:\Sigma A\to Y$ is uniquely determined by these two properties.
\end{remark}

\begin{lemma}\label{lemma.nu}
If $u,v : \Sigma A\to Y$ then $d(u\nu_c,v\nu_c)=u-v$.
\end{lemma}

\begin{proof} If $u,v:\Sigma A\to Y$ then $u\nu_c, v\nu_c: C_c\to Y$ satisfy $u\nu_c i_X=v\nu_c i_X$ and so $w=d(u\nu_c,v\nu_c)$ is defined. Thus, for any $(a,t)\in \Sigma A$ it holds \begin{align*}
w(a,t) &= \begin{cases}
wp_1(a,2t)=u\nu_c i_{CA}(a,2t)=u(a,2t), & \text{if $0\leq t\leq \frac{1}{2}$,} \\[1ex]
wp_2(a,2-2t)=v\nu_c i_{CA}(a,2-2t)=v(a,2-2t), & \text{if $\frac{1}{2}\leq t\leq 1$,}
\end{cases} \\[1ex]
&= (u-v)(a,t).
\end{align*}
Hence, $d(u\nu_c,v\nu_c)=u-v:\Sigma A\to Y$ and the result follows.
\end{proof}

If $f': A'\to X'$, $f : A\to X$ and $u,v: C_{f}\to Y$ satisfy $u i_{X}=v i_{X}$ such that $w=d(u,v):\Sigma A\to Y$ is defined, then it
is clear that $ u\gamma(f',f) i_{X'}=v\gamma(f',f) i_{X'}$ and $w'=d(u\gamma(f',f),v\gamma(f',f)):\Sigma A'\to Y$
is defined as well. To simplify notation, we write simply  $\gamma$  for $\gamma(f',f)$ in the rest of the paper.

We claim that:
\begin{lemma}\label{lemma.d.susp}
$d(u\gamma,v\gamma)=d(u,v)\Sigma\alpha$.
\end{lemma}
\begin{proof}
The separation maps $d(u,v)$ and $d(u \gamma,v \gamma)$ are defined by means of
\[w(a,t)=\begin{cases}
u(a,2t), &\text{if $0\leq t\leq \frac{1}{2}$,}\\[1ex]
v(a,2-2t), & \text{if $\frac{1}{2}\leq t\leq 1$}
\end{cases}\]
for $(a,t)\in\Sigma A$ and
\[w'(a',t)=\begin{cases}
u \gamma(a',2t),   &\text{if $0\leq t\leq \frac{1}{2}$,}\\[1ex]
v \gamma(a',2-2t), &\text{if $\frac{1}{2}\leq t\leq 1$}
\end{cases}\]
for $(a',t)\in \Sigma A'$, respectively.
Then, 
\[(w \Sigma\alpha)(a',t)=\begin{cases}
u(\alpha(a'),2t),   &\text{if $0\leq t\leq \frac{1}{2}$,}\\[1ex]
v(\alpha(a'),2-2t), &\text{if $\frac{1}{2}\leq t\leq 1$}
\end{cases}\]
for $(a',t)\in \Sigma A'$. Because $\gamma  i_{CA'}=i_{CA}  C\alpha$, we derive that $w'=w \Sigma\alpha$ and the proof follows.
\end{proof}

Denote by $\tilde\sigma:\Sigma^2A\xrightarrow{\approx} \Sigma^2A$
the homeomorphism  defined by $(s,(s',a))\mapsto (s',(s,a))$ and notice that $f\tilde{\sigma}=-f$ for any map $f:\Sigma^2A\to X$. Since $u'=(\Sigma u)\tau_1$ and $v'=(\Sigma v)\tau_1: C_{\Sigma f} \to \Sigma Y$ satisfy $u' i_{\Sigma X}=v' i_{\Sigma X}$ the map
$w'=d(u',v'):\Sigma^2A\to \Sigma Y$ is defined and satisfies $w' p_1'=u' i_{C\Sigma A}$ and $w' p_2'=v' i_{C\Sigma A}$,
where $p_1',p_2':C\Sigma A \to \Sigma^2A$ are maps as in Remark~\ref{rem.sep.map.proj}.

Then, we can state (cf.\ \cite[Lemma~4.1]{hardie3}):

\begin{lemma}$d((\Sigma u)\tau_1,(\Sigma v) \tau_1)=(\Sigma d(u,v)) \tilde\sigma$.
\label{lemma.d.permut}\end{lemma}

\begin{proof}Just observe that $(\Sigma p_1) \tau_2=\tilde{\sigma}p_1'$ and then the equalities $w'  p_1'=(\Sigma u) \tau_1  i_{C\Sigma A}=\Sigma u  (\Sigma i_{CA}) \tau_2=(\Sigma w  p_1) \tau_2=(\Sigma w)  \tilde\sigma p_1'$ and similarly $w' p_2'=(\Sigma w)  \tilde\sigma p_2'$ imply the result.
\end{proof}

Next, if  $u,v:T_0(\underline{\Sigma A})\to Y$ coincide on $T_1(\underline{\Sigma A})$ then there exist the separation maps
$d(u\nu_{\omega_n},v\nu_{\omega_n}):\Sigma^n\Lambda(\underline{A})\to Y$ and $d(\Sigma (u\nu_{\omega_n})\tau_1,\Sigma (v\nu_{\omega_n})\tau_1):\Sigma^{n+1}\Lambda(\underline{A})\to \Sigma Y$.
Thus, the diagram~\eqref{d}, Lemmas~\ref{lemma.d.susp} and \ref{lemma.d.permut} lead to:
\begin{corollary}If $\gamma=\gamma(c,\Sigma\omega_n):C_c\to C_{\Sigma \omega_n}$ then
\[d(\Sigma (u\nu_{\omega_n})\tau_1\gamma,\Sigma (v\nu_{\omega_n})\tau_1\gamma)=(\Sigma d(u,v)) \tilde\sigma\Sigma\lambda^{-1}\]
in $\pi(\Sigma\Lambda(\underline{\Sigma A}),\Sigma Y)$.
\label{cor.gamma}\end{corollary}

\medskip

To state the main result of this section, we recall from \cite{hardie3} the notion of a map of a strongly type.
\begin{definition}\label{def.strong.type}
Given $\alpha_i:  A_i\to J_{m_i}(X)$ for $i=1,\dots,n$, we say that $F:T_0(\underline{A}) \to J(X)$ is \textit{strongly of type}
$(\alpha_1,\dots,\alpha_n)^{k}$ if its image is contained in $J_{k}(X)$ and coincides on $T_1(\underline{A})$ with
$F'=\mu_{\mathbf{m}}(X) T_0(\underline{\alpha}): T_0(\underline{A})\to J_{m}(X)$, where $\underline{\alpha}=(\alpha_1,\dots,\alpha_n)$, $\mathbf{m}=(m_1,\dots,m_n)$ and $m=m_1+\dots+m_n$.
\end{definition}

\begin{lemma}[{cf. \cite[Theorem~1.8]{hardie3}}]Let $F:T_0(\underline{\Sigma A})\to J(X)$ be a map strongly of some type.
Then $\phi_X d(F'\nu_{\omega_n},F\nu_{\omega_n})\lambda^{-1}=\tilde c(F)$.
\label{LL}\end{lemma}

\begin{proof}
Let $F':T_0(\underline{\Sigma A})\to J(X)$ be a map as in Definition~\ref{def.strong.type} and let $\gamma=\gamma(c,\Sigma\omega_n):C_c\to C_{\Sigma\omega_n}$ as in Corollary~\ref{cor.gamma}. By~\eqref{eq.tau} and Proposition~\ref{p}, $(\Sigma\nu_{\omega_n})\tau_1\gamma\nu_c^{-1}=\tau=\delta$. Then,
\begin{multline*}
\psi_{X}(\Sigma d(F'\nu_{\omega_n},F\nu_{\omega_n}))\tilde{\sigma}\Sigma\lambda^{-1} \\
\begin{aligned}
&= \psi_{X} d(\Sigma (F'\nu_{\omega_n})\tau_1,\Sigma( F\nu_{\omega_n})\tau_1)\Sigma\lambda^{-1} \\
&= \psi_{X}d(\Sigma (F'\nu_{\omega_n}) \tau_1\gamma, \Sigma (F\nu_{\omega_n})\tau_1\gamma) \\
&= d(\psi_{X}(\Sigma F')(\Sigma\nu_{\omega_n})\tau_1\gamma\nu_c^{-1}\nu_c,\psi_X(\Sigma F)(\Sigma\nu_{\omega_n})\tau_1\gamma\nu_c^{-1}\nu_c)\\
&= d(c(F')\nu_c,c(F)\nu_c)\\
&= c(F')-c(F).
\end{aligned}
\end{multline*}
Since $(\Sigma d(F'\nu_{\omega_n},F\nu_{\omega_n}))\tilde{\sigma}=-\Sigma d(F'\nu_{\omega_n},F\nu_{\omega_n})$ and $c(F')=\ast$ (\cite[Corollary 3.4]{hardie3}), we get
$\psi_{X} (\Sigma d(F'\nu_{\omega_n},F\nu_{\omega_n}))\Sigma\lambda^{-1}=c(F)$
which certainly implies \[\phi_X d(F'\nu_{\omega_n},F\nu_{\omega_n})\lambda^{-1}=\tilde c(F).\]
This completes the proof.
\end{proof}

We finish this section with a generalization of \cite[Theorem~1.10]{hardie3} as the main result.

Let $q_m:(J_m(X),J_{m-1}(X))\to (X^{\wedge m},\ast)$ be the quotient map and denote by $q : J(X)\to J(X^{\wedge m})$ its combinatorial extension. Given $\alpha_i:\Sigma A_i\to J_{m_i}(X)$ for $i=1,\dots,n$, we obtain the maps \[ \Sigma A_i \xrightarrow{\alpha_i} J_{m_i}(X)\xrightarrow{q_{m_i}} X^{\wedge m_i} \] and the suspension of their smash products leads to \[\Sigma(q_{m_1}\alpha_1\wedge\dots\wedge q_{m_n}\alpha_n): \Sigma^{n+1}\Lambda(\underline{A}) \to \Sigma X^{\wedge m},\] where $m=m_1+\dots+m_n$.

\begin{theorem}\label{thm:main}If $\alpha_i\in \pi(\Sigma A_i,J_{m_i}(X))$ for $i=1,\dots,n$ and $F:T_0(\underline{\Sigma A})\to J(X)$ is strongly of type $(\alpha_1,\dots,\alpha_n)^{m-1}$ then \[h_m(c(F))=\Sigma(q_{m_1}\alpha_1\wedge\dots\wedge q_{m_n}\alpha_n). \]
\end{theorem}

\begin{proof} In view of~\eqref{forr} and Lemma~\ref{LL}, we have
\begin{align*}
h_m(c(F)) &= \psi_{X^{\wedge m}} \Sigma (q\phi_X^{-1})\Sigma \tilde{c}(F) \\
&= \psi_{X^{\wedge m}} \Sigma (q\phi_X^{-1})\Sigma (\phi_X d(F'\nu_{\omega_n},F\nu_{\omega_n})\lambda^{-1}) \\
&= \psi_{X^{\wedge m}} (\Sigma d(qF'\nu_{\omega_n},qF\nu_{\omega_n}))\Sigma\lambda^{-1}.
\end{align*}
Since $F(T_0(\underline{\Sigma A}))\subseteq J_{m-1}(X)$ and $q$ is the combinatorial extension of $q_m$ then $qF\nu_{\omega_n}=\ast$. So $w=d(qF'\nu_{\omega_n},qF\nu_{\omega_n}):\Sigma^n\Lambda(\underline{A})\to J(X^{\wedge m})$ satisfies $wp_1=qF'\Omega_n$ and $wp_2=\ast$. Consequently, there exists $w'= w$ such that $w'p=qF'\Omega_n$, where $p: C\Sigma^{n-1}\Lambda(\underline{A}) \to \Sigma^n\Lambda(\underline{A})$ is the map from diagram~\eqref{diag.lambda}.

Next, consider the following commutative diagram
\[\begin{aligned}
\xymatrix@C=1.25cm@R=1cm{
C\Sigma^{n-1}\Lambda(\underline{A}) \ar[r]^-{\Omega_n} \ar[d]_{p} & T_0(\underline{\Sigma A}) \ar[r]^{F'} \ar[d]_{s} & J_m(X) \ar@{^(->}[r] \ar[d]^{q_m}   & J(X) \ar[d]^{q}\\
\Sigma^n\Lambda(\underline{A}) \ar[r]^{\lambda}
\ar `d/8pt[r] `[rrr]^{w'} [rrr]
& \Lambda(\underline{\Sigma A}) \ar[r]^{F''} & X^{\wedge m} \ar@{^(->}[r]^{j_{X^{\wedge m}}} & J(X^{\wedge m})\rlap{,}
}
\end{aligned}\]
where $F''$ is determined by $F''s=q_mF'$ and the other maps were already defined in the text.
Hence, $h_m(c(F))=\psi_{X^{\wedge m}} (\Sigma d(qF'\nu_{\omega_n},qF\nu_{\omega_n}))\Sigma\lambda^{-1}=\psi_{X^{\wedge m}} (\Sigma j_{X^{\wedge m}}F''\lambda)\Sigma\lambda^{-1}=
\psi_{X^{\wedge m}} (\Sigma j_{X^{\wedge m}}F'')=\Sigma F''$.

Since $F'=\mu_{\mathbf{m}}(X) T_0(\underline{\alpha})$, the commutativity of the diagram
\[\xymatrix@C=2cm{
& T_0(\underline{J_{\mathbf{m}}(X)}) \ar[d]^{T_0(\underline{q_{\mathbf{m}}})} \ar[r]^-{\mu_{\mathbf{m}}(X)} & J_m(X) \ar `d/0pt[dd] [ddl]_-{q_m} \\
T_0(\underline{\Sigma A}) \ar[r]^{T_0(\underline{q_{\mathbf{m}}})T_0(\underline{\alpha})} \ar[d]_{s} \ar `u/0pt[u]  [ru]^-{T_0(\underline{\alpha})} & T_0(\underline{X^{\wedge\mathbf{m}}}) \ar[d]_{s} & \\
\Lambda(\underline{\Sigma A}) \ar[r]^{F''} & X^{\wedge m} &
}\]
where $\underline{q_{\mathbf{m}}}=(q_{m_1},\dots,q_{m_n})$ finishes the proof.
\end{proof}

Let $\alpha_i : \Sigma A_i\to X$ be maps for $i=1,\dots,n$. If $\langle\alpha_1,\dots,\alpha_n\rangle=0$ then there is a map
$F : T_0(\underline{\Sigma A})\to J(X)$ strongly of type  $(\alpha_1,\dots,\alpha_n)^{n-1}$. Thus, the generalized Hopf
construction yields a map $c(F) : \Sigma\Lambda(\underline{\Sigma A})\to \Sigma X$ and Theorem~\ref{thm:main} supplies
a criterion to check if the map $c(F)$ is non-trivial.

\subsection{Miscellanea}\label{subsec.misc}
Fix $n\geq 2$ and suppose $A_i=A$ for $i=1,\dots,n$. Denote by $\mathrm{id}_{\Sigma A}: \Sigma A\to J_1(\Sigma A)$
the identity map and by $\rho_n(\Sigma A):T_0(\underline{\Sigma A})\to J_n(\Sigma A)$ the quotient map. For the sequences
$\mathbf{m}=(1,\overset{\times n}{\dotsc{}},1)$ and $\underline{\mathrm{id}_{\Sigma A}}=(\mathrm{id}_{\Sigma A},\overset{\times n}{\dotsc{}},\mathrm{id}_{\Sigma A})$
the map $\rho_n(\Sigma A)$ factorizes as \[\rho_n(\Sigma A): T_0(\underline{\Sigma A})\xrightarrow{T_0(\underline{\mathrm{id}_{\Sigma A}})} T_0(\underline{J_{\mathbf{m}}(\Sigma A)}) \xrightarrow{\mu_{\mathbf{m}}(\Sigma A)} J_n(\Sigma A) \] and restricts to
$\rho_n(\Sigma A)_{\mid T_1(\underline{\Sigma A})} : T_1(\underline{\Sigma A}) \to J_{n-1}(\Sigma A)$. This leads to a (up to homotopy) pushout
\[\xymatrix@C=2.5cm{T_1(\underline{\Sigma A})\ar[r]^-{\rho_n(\Sigma A)_{\mid T_1(\underline{\Sigma A})}} \ar@{^{(}->}[d] & J_{n-1}(\Sigma A) \ar@{^{(}->}[d] \\
T_0(\underline{\Sigma A})\ar[r]^{\rho_n(\Sigma A)} & J_n(\Sigma A) \rlap{.} } \]
Taking into account the pushout diagram~\eqref{diag.omega}, we get a (up to homotopy) pushout
\begin{equation}\label{dia}
\begin{gathered}
\xymatrix@C=2.5cm{\Sigma^{n-1}\Lambda(\underline{A})\ar[r]^-{\langle\mathrm{id}_{\Sigma A},\overset{\times n}{\dotsc{}},\mathrm{id}_{\Sigma A}\rangle}
\ar@{^{(}->}[d] & J_{n-1}(\Sigma A) \ar@{^{(}->}[d] \\
C\Sigma^{n-1}\Lambda(\underline{A})\ar[r]^{\rho_n(\Sigma A)\Omega_n} & J_n(\Sigma A)}
\end{gathered}
\end{equation}
which yields the following result (\cite[Proposition~1.1.1]{wu}):

\begin{proposition} Let $A$ be a pointed space. Then there is a \textup{(\textit{functorial})} cofibre sequence
\[\Sigma^{n-1}\Lambda (\underline{A})\xrightarrow{\rho_n(\Sigma A)_{\mid T_1(\underline{\Sigma A})}\omega_n} J_{n-1}(\Sigma A)\hookrightarrow J_n(\Sigma A).\]
Thus, the cofibre sequence
\[J_{n-1}(\Sigma A)\hookrightarrow J_n(\Sigma A)\to \Lambda(\underline{\Sigma A})\] is principal.
\end{proposition}

Given simply-connected co-$H$-spaces $X_1,X_2$, Gray \cite{gray} has defined the Theriault product $X_1\circ X_2$
being a retraction of $\Sigma(\Omega X_1\wedge \Omega X_2)$. If
\[X_1\circ X_2 \xrightarrow{\zeta} \Sigma(\Omega X_1\wedge \Omega X_2) \xrightarrow{\kappa} X_1\circ X_2 \]
are maps with $\kappa\zeta=\mathrm{id}_{X_1\circ X_2}$ then the homotopy fibration
\[\Sigma(\Omega X_1\wedge \Omega X_2)\xrightarrow{w} X_1\vee X_2\hookrightarrow X_1\times X_2\]
determines a natural transformation
\[w_2 : X_1\circ X_2 \xrightarrow{\zeta} \Sigma(\Omega X_1\wedge \Omega X_2) \xrightarrow{w} X_1\vee X_2\]
generalizing the Whitehead product map.

We make use of the Theriault product to define the higher order Gray--Whitehead product for co-$H$-spaces $X_1,\dots,X_n$. As in \cite[Section~3]{gt}, we start recalling that Porter \cite[Theorem~1]{porter3} has shown that there is a homotopy fibration
\[\Sigma^{n-1}\Lambda(\underline{\Omega X})\to T_1(\underline{X})\hookrightarrow T_0(\underline{X}).\]
Because $X_i$ are co-$H$-spaces, there are coretractions $\nu_i : X_i\to \Sigma\Omega X_i$ for $i=1,\dots,n$. Define the higher Gray--Whitehead product map
\[w_n:(X_1\circ X_2) \wedge X_3\wedge \dots\wedge X_n\to T_1(\underline{X})\] as the composite  \begin{multline*}
(X_1\circ X_2)\wedge X_3\wedge \dots \wedge X_n \xrightarrow{\zeta\wedge\nu_3\wedge\dots\wedge\nu_n}\\
\Sigma(\Omega X_1\wedge \Omega X_2)\wedge \Sigma\Omega X_3 \wedge\dots\wedge \Sigma\Omega X_n =
\Sigma^{n-1}\Lambda(\underline{\Omega X})\to T_1(\underline{X}).
\end{multline*}

Notice that, by means of the above, basic results presented in previous sections might be generalized replacing suspended spaces by co-$H$-spaces.

Applying \cite[Theorem~1]{gray}, the inductive procedure shows that \[(X_1\circ X_2)\wedge X_3\wedge \dots \wedge X_n=\Sigma^{n-2}(X_1\circ \dots \circ X_n)\]
and consequently \[w_n : \Sigma^{n-2}(X_1 \circ \dots \circ X_n)\to  T_1(\underline{X}).\]

Now, we make use of the above for the spherical interior Whitehead product.
Let $\alpha_i :\mathbb{S}^{k_i}\to\mathbb{S}^m$ be maps with $m\ge 2$, $k_i\ge 2$ for $i=1,\dots,n$ and $n\ge 2$.
Then, $\langle\alpha_1,\dots,\alpha_n\rangle=0$ in $\pi_{k_1+\dots+k_n-1}(J_{n-1}(\mathbb{S}^m))$ yields a map
$F : \mathbb{S}^{k_1}\times\dots\times\mathbb{S}^{k_n}\to J(\mathbb{S}^m)$ and the generalized Hopf  construction
leads to possibly a non-trivial map $c(F) : \mathbb{S}^{k_1+\dots+k_n+1}\to \mathbb{S}^{m+1}$.

Next, write $\iota_n=\mathrm{id}_{\mathbb{S}^n}$, $\iota_{m,n} :  \mathbb{S}^n\hookrightarrow J_m(\mathbb{S}^n)$
for the canonical inclusion maps and $\eta_n \in\pi_{n+1}(\mathbb{S}^n)$ for generators with $n\geq 2$.

\begin{proposition}\label{pp}
\begin{enumerate}[label={\textup{(\arabic*)}}]\thmenumhspace

\item\label{pp.1} The element $\langle\iota_n,\overset{\times m}{\dotsc{}},\iota_n\rangle$ is of infinite order provided
$n$ is odd and $m\neq 2$ or $n$ is even;

\item\label{pp.2} $\pi_{mn-1}(J_{m-1}(\mathbb{S}^n))\approx\mathbb{Z}\oplus\pi_{mn}(\mathbb{S}^{n+1})$ and $\langle\iota_n,\overset{\times m}{\dotsc{}},\iota_n\rangle$ is a generator of the infinite cyclic group;

\item\label{pp.3} $[\iota_{m-2,n},\langle\iota_n,\overset{\times (m-1)}{\dots\dotsc{}},\iota_n\rangle]=0$ if and only if $n=2$ and $m$ is an odd prime; this element has order $m$ otherwise;

\item\label{pp.4} $\langle\eta_2\eta_3\eta_4,\iota_2,\iota_2,\iota_2\rangle=0$ in  $\pi_{10}(J_3(\mathbb{S}^2))$;

\item\label{pp.5} $\langle\eta_2\eta_3,\iota_2,\iota_2,\iota_2\rangle=0$ in $\pi_9(J_3(\mathbb{S}^2))$;

\item\label{pp.6} $\langle\eta_3,\eta_3,\iota_3\rangle=0$ in $\pi_{10}(J_2(\mathbb{S}^3))$.
\end{enumerate}
\end{proposition}
\begin{proof}
\ref{pp.1}: follows from \cite{baues1,hardie1,shar}.

\ref{pp.2}: follows from \cite{baues1,shar}.

\ref{pp.3}: follows from \cite{baues1,hardie1}.

\ref{pp.4}: follows from \cite{hardie1}.

\ref{pp.5}--\ref{pp.6}: we make use of the isomorphisms $\pi_9(\mathbb{S}^2)\approx\mathbb{Z}_3$ and $\pi_{10}(\mathbb{S}^3)\approx\mathbb{Z}_{15}$
and then follow \textit{mutatis mutandis} the proof of \ref{pp.4}.
\end{proof}

Proposition~\ref{pp}\ref{pp.3} implies the existence of a map $F$ strongly
of type $(\iota_{m-2,2},\langle\iota_2,\overset{\times (m-1)}{\dots\dotsc{}},\iota_2\rangle)^{m-2}$ for an odd prime $m$
which yields, in view of \cite[Theorem~1.5]{hardie1}, an element $c(F)$ of order $m$ in $\pi_{2m}(\mathbb{S}^3)$.
Further, Proposition~\ref{pp}\ref{pp.4}--\ref{pp.6} implies also the existence of maps $F$ strongly of types
$(\eta_2\eta_3\eta_4,\iota_2,\iota_2,\iota_2)^3$, $(\eta_2\eta_3,\iota_2,\iota_2,\iota_2)^3$
and $(\eta_3,\eta_3,\iota_3)^2$, respectively. Then, as in \cite[Corollary~1.4]{hardie1}
in view of Theorem~\ref{thm:main}, we obtain non-zero elements $c(F)$ of $\pi_{12}(\mathbb{S}^3)$,
$\pi_{11}(\mathbb{S}^3)$ and $\pi_{12}(\mathbb{S}^4)$, respectively.

\medskip

Write $SP_n(X)$ for the $n$\textsuperscript{th} stage of the symmetric power of a space $X$ for $n\ge 1$ and $SP(X)=\operatorname{\underrightarrow{\lim}}SP_n(X)$.
Because of the $H$-structure on $SP(X)$, the inclusion map $X\hookrightarrow SP(X)$ extends to a map
\[u(X) : J(X)\to SP(X)\] which leads to a sequence of maps $u_n(X): J_n(X)\to SP_n(X)$ for $n\ge 1$.
Taking $u_{n-1}(\Sigma A)\langle\mathrm{id}_{\Sigma A},\overset{\times n}{\dotsc{}},\mathrm{id}_{\Sigma A}\rangle$
we get a (up to homotopy) pushout
\begin{equation}
\begin{gathered}
\xymatrix@C=3.75cm{\Sigma^{n-1}\Lambda(\underline{A})\ar[r]^-{u_{n-1}(\Sigma A)\langle\mathrm{id}_{\Sigma A},\overset{\times n}{\dotsc{}},\mathrm{id}_{\Sigma A}\rangle}
\ar@{^{(}->}[d] & SP_{n-1}(\Sigma A) \ar@{^{(}->}[d] \\
C\Sigma^{n-1}\Lambda(\underline{A})\ar[r]^-{u_n(\Sigma A)\rho_n(\Sigma A)\Omega_n}& SP_n(\Sigma A) \rlap{.} }
\end{gathered}\label{sym}
\end{equation}
Notice that armed with the diagrams~\eqref{dia} and~\eqref{sym}, the sequence of maps $u_n(\Sigma A): J_n(\Sigma A)\to SP_n(\Sigma A)$ above
for $n\ge 1$ might be derived by the inductive procedure as well.

In particular, this yields a sequence of maps $u_n(\mathbb{S}^1) : J_n(\mathbb{S}^1)\to SP_n(\mathbb{S}^1)$ for $n\ge 1$. On the other hand, the Segre map \[\mathbb{R}P^m\times\mathbb{R}P^n\to\mathbb{R}P^{(m+1)(n+1)-1}\]
given by $([x_0\mathbin{:}\dotsb\mathbin{:}x_m],[y_0\mathbin{:}\dotsb\mathbin{:}y_n])\mapsto [x_0y_0\mathbin{:}\dotsb\mathbin{:}x_iy_j\mathbin{:}\dotsb\mathbin{:}x_my_n]$
leads to an $H$-structure on the infinite real projective space $\mathbb{R}P^\infty$.
Thus, the inclusion map $\mathbb{S}^1=\mathbb{R}P^1\hookrightarrow\mathbb{R}P^\infty$ extends to a map
\[v(\mathbb{S}^1) : J(\mathbb{S}^1)=\Omega\mathbb{S}^2\to \mathbb{R}P^\infty\] which leads to
a sequence of maps $v_n(\mathbb{S}^1) : J_n(\mathbb{S}^1)\to\mathbb{R}P^n$ for $n\ge 1$. Further, the abelian $H$-structure on $\mathbb{R}P^\infty$ yields the factorization
\[\xymatrix{\mathbb{S}^1 \ar@{^{(}->}[r] \ar@{^{(}->}[d] & \mathbb{R}P^\infty \\
SP(\mathbb{S}^1) \ar@{-->}[ur]} \]
which in turn implies the sequence of maps $\varphi_n(\mathbb{S}^1) : SP_n(\mathbb{S}^1)\to \mathbb{R}P^n$ with commutative diagrams
\[\xymatrix@C=1.5cm{J_n(\mathbb{S}^1)\ar[r]^-{v_n(\mathbb{S}^1)}\ar[d]_{u_n(\mathbb{S}^1)}&\mathbb{R}P^n \\
SP_n(\mathbb{S}^1)\ar[ur]_{\varphi_n(\mathbb{S}^1)}}\]
for $n\ge 1$. Because, by means of~\cite{morton}, the space $SP_n(\mathbb{S}^1)$ has the homotopy type of the circle $\mathbb{S}^1$,
the induced homomorphisms $\pi_k(v_n(\mathbb{S}^1))$ are trivial for $k>1$ and $n\ge 1$.

Let now $\mathbb{C}P^n$ be the complex projective $n$-space. By the projective Vi\`ete's Theorem (see e.g., \cite{arnold}), it holds $SP_n(\mathbb{S}^2)=\mathbb{C}P^n$
and we get a sequence of maps $u_n(\mathbb{S}^2) : J_n(\mathbb{S}^2)\to \mathbb{C}P^n$ for $n\ge 1$.
Notice that these maps are also determined by the $H$-structure (settled e.g., by the Segre map) on $\mathbb{C}P^\infty$ and the factorization
\[\xymatrix{\mathbb{S}^2 \ar@{^{(}->}[r] \ar@{^{(}->}[d] & \mathbb{C}P^\infty \\
J(\mathbb{S}^2)\rlap{.} \ar@{-->}[ur]} \]
Write $\gamma_n :\mathbb{S}^{2n+1}\to\mathbb{C}P^n$ for the quotient map and $j_n(\mathbb{C}) : \mathbb{S}^2=\mathbb{C}P^1\hookrightarrow\mathbb{C}P^n$
for the canonical inclusion. It is known from \cite[Corollary~4.4]{ark} and \cite[Corollary~2]{porter2} that the set
$[j_n(\mathbb{C}),\overset{\times (n+1)}{\dots\dotsc{}},j_n(\mathbb{C})]$ of $(n+1)$\textsuperscript{th} order Whitehead products contains a single element which is equal
to $(n+1)\mathpunct{!}\gamma_n$. Consequently, by means of~\eqref{eq.int.nat}, for $n\ge 1$ the map $u_n(\mathbb{S}^2) : J_n(\mathbb{S}^2)\to\mathbb{C}P^n$
satisfies \[u_n(\mathbb{S}^2)\langle\iota_2,\overset{\times (n+1)}{\dots\dotsc{}},\iota_2\rangle=[j_n(\mathbb{C}),\overset{\times (n+1)}{\dots\dotsc{}},j_n(\mathbb{C})]=(n+1)\mathpunct{!}\gamma_n.\]

Let now $\mathbb{H}$ be the quaternionic algebra and $j_n(\mathbb{H}) : \mathbb{S}^4=\mathbb{H}P^1\hookrightarrow\mathbb{H}P^n$ the canonical inclusion.
Then, by \cite[Remark~4.9(iii)]{marek-thiago}, the higher order Whitehead product $[j_n(\mathbb{H}),\overset{\times (n+1)}{\dots\dotsc{}},j_n(\mathbb{H})]=\emptyset$
for $n\ge 2$. Hence, an existence of a map $u_n(\mathbb{S}^4) : J_n(\mathbb{S}^4)\to\mathbb{H}P^n$ with properties as above
leads to a contradiction $u_n(\mathbb{S}^4)\langle\iota_4,\overset{\times (n+1)}{\dots\dotsc{}},\iota_4\rangle\in[j_n(\mathbb{H}),\overset{\times (n+1)}{\dots\dotsc{}},j_n(\mathbb{H})]=\emptyset$.

Let $\alpha_i :\mathbb{S}^{k_i}\to\mathbb{S}^1$ be maps with $k_i\ge 1$ for $i=1,\dots,n$ and $n\ge 2$. If $k_{i_0}> 1$ for some $1\le i_0\le n$ then $\alpha_{i_0}=0$ and certainly $\langle\alpha_1,\dots,\alpha_n\rangle=0$ in $\pi_{k_1+\dots+k_n-1}(J_{n-1}(\mathbb{S}^1))$.


Next, consider maps $\alpha_i :\mathbb{S}^{k_i}\to X$  with $k_i\ge 2$ for $i=1,\dots,n$ and $n\ge 2$, where $X=\mathbb{S}^2$ or $\mathbb{R}P^2$.
Then, $u_{n-1}(\mathbb{S}^2)\langle\alpha_1,\dots,\alpha_n\rangle=0$ provided $k_1+\dots+k_n<2n$. Further, by~\cite[Theorem~2]{arnold}, it holds $SP_n(\mathbb{R}P^2)=\mathbb{R}P^{2n}$. Hence,
$u_{n-1}(\mathbb{R}P^2)\langle\alpha_1,\dots,\alpha_n\rangle=0$ provided $k_1+\cdots+k_n<2n-1$.

We observe that \cite[Theorem~2]{arnold} also gives sequences of maps $A_n(\mathbb{S}^2)$ and $A_n(\mathbb{R}P^2)$ for $n\geq 1$, fitted together by the commutative diagrams
\[ \xymatrix@C=1.5cm{
J_n(\mathbb{S}^2) \ar[d]_{A_n(\mathbb{S}^2)} \ar[r] & J_n(\mathbb{R}P^2) \ar[d]^{A_n(\mathbb{R}P^2)} \\
\mathbb{C}P^n \ar[r] & \mathbb{R}P^{2n} \rlap{.}
} \]

In this sense, we close the paper with:
\begin{conjecture}Let $\alpha_i :\mathbb{S}^{k_i}\to X$ be maps with $k_i\ge 2$ for $i=1,\dots,n$ and $n\ge 2$, where $X=\mathbb{S}^2$ or $\mathbb{R}P^2$.
If $k_{i_0}>2$ for some $1\le i_0\le n$ then $\langle\alpha_1,\dots,\alpha_n\rangle=0$ in $ \pi_{k_1+\dots+k_n-1}(J_{n-1}(X))$.
\end{conjecture}

\end{document}